\documentclass[11pt,a4paper]{amsart}
\usepackage{amssymb, amstext, amscd, amsmath}

\usepackage{url}
\usepackage{amsfonts}
\usepackage{lscape}
\usepackage{amsthm}
\usepackage{amsfonts}
\usepackage{array}
\usepackage{eucal}
\usepackage{latexsym}
\usepackage{mathrsfs}
\usepackage{textcomp}
\usepackage{verbatim}

\newtheorem{thm}{Theorem}[section]

\newtheorem{claim}[thm]{Claim}

\newtheorem{con}[thm]{Conjecture}
\newtheorem{cor}[thm]{Corollary}

\newtheorem{lem}[thm]{Lemma}
\newtheorem{prop}[thm]{Proposition}

\numberwithin{equation}{section}

\newcommand{\bC}{{\mathbb{C}}}

\newcommand{\bQ}{{\mathbb{Q}}}
\newcommand{\bR}{{\mathbb{R}}}

\newcommand{\bZ}{{\mathbb{Z}}}

\newcommand{\bK}{{\mathbb{K}}}
  \newcommand{\C}{{\mathcal{C}}}
  
  \newcommand{\E}{{\mathcal{E}}}

  \newcommand{\M}{{\mathcal{M}}}
  \newcommand{\N}{{\mathcal{N}}}
\renewcommand{\O}{{\mathcal{O}}}
\renewcommand{\P}{{\mathcal{P}}}

\renewcommand{\S}{{\mathcal{S}}}
  \newcommand{\T}{{\mathcal{T}}}

  \newcommand{\Y}{{\mathcal{Y}}}

\newcommand{\rank}{\operatorname{rank}}
\newcommand{\fract}{\mathrm{Fract}}
\newcommand{\td}{\mathrm{td}\,}

\begin{document}

\title[Global Rigidity on Surfaces]{Necessary Conditions for the Generic Global Rigidity of Frameworks on Surfaces}
\author[B. Jackson]{B. Jackson}
\address{School of Mathematical Sciences\\ Queen Mary, University of London\\
E1 4NS \\ U.K. }
\email{b.jackson@qmul.ac.uk}
\author[T. A. McCourt]{T. A. McCourt}
\address{Heilbronn Institute for Mathematical Research\\ School of Mathematics\\ University of Bristol\\ Bristol\\
BS8 1TW \\ U.K. }
\email{tom.mccourt@bristol.ac.uk}
\author[A. Nixon]{A. Nixon}
\address{Heilbronn Institute for Mathematical Research\\ School of Mathematics\\ University of Bristol\\ Bristol\\
BS8 1TW \\ U.K. }
\email{tony.nixon@bristol.ac.uk}


\begin{abstract}
A result due in its various parts to Hendrickson, Connelly, and Jackson and Jord\'an, provides a purely combinatorial characterisation of global rigidity for generic bar-joint frameworks in ${\bR}^2$. The analogous conditions are known to be insufficient to characterise generic global rigidity in higher dimensions.
Recently Laman-type characterisations of rigidity have been obtained for generic frameworks in $\bR^3$ when the vertices are constrained to lie on various surfaces, such as the cylinder and the cone.
In this paper we obtain analogues of Hendrickson's necessary conditions for the global rigidity of generic frameworks on the cylinder, cone and ellipsoid.
\end{abstract}

\maketitle

\section{Introduction}

A bar-joint framework in Euclidean space $\bR^d$ is a geometric realisation of the vertices of a graph with the edges considered as (fixed length) bars between them.
Such a framework is said to be \emph{rigid} if there is no non-trivial continuous motion of the vertices in $\bR^d$ which maintains bar-lengths, and is said to be flexible if it is not rigid. It is \emph{redundantly rigid} if it remains rigid after deleting any single edge. A foundational theorem of Laman \cite{Lam}, obtained in 1970,  asserts that the rigidity of a generically positioned framework in $\bR^2$ depends only on the underlying graph and, furthermore, that these graphs are characterised in terms of a simple counting condition.
Finding an analogous characterisation for the rigidity of generic frameworks in $\bR^3$ is an important open problem.

Formally a framework $(G,p)$ in $\bR^d$ is the combination of a finite graph $G=(V,E)$ and a map $p:V\rightarrow \bR^d$.
Two frameworks $(G,p)$ and $(G,q)$ are said to be \emph{equivalent} if $\|p(v)-p(u)\|=\|q(v)-q(u)\|$ for all pairs of adjacent vertices $u,v\in V$. More strongly they are said to be \emph{congruent} if $\|p(v)-p(u)\|=\|q(v)-q(u)\|$ holds for all pairs of vertices $u,v\in V$. A framework $(G,p)$ is \emph{globally rigid} if every framework $(G,q)$ equivalent to $(G,p)$ is also congruent to $(G,p)$.

Hendrickson derived the following necessary conditions for generic global rigidity in $\bR^d$.\footnote{More precisely, Hendrickson proved the weaker result that {\em almost all} globally rigid frameworks in $\bR^d$ are redundantly rigid. His proof technique can be extended to cover all generic frameworks and this extension has been taken to be implicit in the literature.}

\begin{thm}[\cite{Hdk}]\label{thm:hennecessary}
Let $(G,p)$ be a generic globally rigid framework in $\bR^d$. Then $G$ is a complete graph on at most $d+1$ vertices or $G$ is $(d+1)$-connected and $(G,p)$ is redundantly rigid in $\bR^d$.
\end{thm}

In the case of 2-dimensional frameworks these conditions are also sufficient.

\begin{thm}[\cite{Hdk}, \cite{Con} and \cite{J&J}]\label{thm:globalplane}
Let $(G,p)$ be a generic framework in $\bR^2$. Then $(G,p)$ is globally rigid if and only if either $G$ is a complete graph on at most three vertices or $G$ is $3$-connected and $(G,p)$ is redundantly rigid.
\end{thm}

\noindent Sufficiency follows by combining a geometric result due to Connelly and a combinatorial construction of Jackson and Jord\'an.

Both rigidity and global rigidity have far reaching applications. In particular, determining when a framework has a unique realisation up to congruence has applications in robotics \cite{Robotics} and in sensor networks \cite{SensorNetworks}.

Attention has recently been given to frameworks in $\bR^3$ whose vertices are constrained to lie on $2$-dimensional surfaces and analogues of Laman's theorem have been obtained for a variety of surfaces \cite{NOP}, \cite{NOP2}. In this paper we consider the global rigidity of frameworks on these surfaces and obtain analogues of Hendrickson's necessary conditions for generic global rigidity. In the case when the surface is a sphere, Connelly and Whiteley \cite{C&W} proved that a generic framework $(G,p)$ on the sphere is globally rigid if and only if a corresponding generic framework $(G,q)$ is globally rigid in the plane. We focus our attention on cylinders, cones and ellipsoids. We also include the sphere as it is covered by our proof technique and provides a complete proof that redundant rigidity is a necessary condition for generic global rigidity on the sphere, and hence also in the plane.

We conclude the introduction by giving a brief outline of the proof of our main result, that redundant rigidity is a necessary condition for generic global rigidity. We adopt a similar approach to that of \cite{Hdk}. We consider the motion of the framework that results from deleting a non-redundant edge  $e$
from a generic rigid framework $(G,p_0)$ on a surface $\M$.  We obtain a series of geometric results in Sections \ref{sec:alggeom}-\ref{sec:regular} that enable us to
 show in Section \ref{sec:global} that this motion is diffeomorphic to a circle. We then use genericity to prove that the motion reaches a framework $(G-e,p_1)$ such that $(G,p_1)$ is equivalent but not congruent to $(G,p_0)$. We have to overcome a slight technical difficulty for surfaces which does not arise in \cite{Hdk}. We actually show in
 Section \ref{sec:global} that there is no isometry of $\M$ which maps $(G,p_0)$ onto $(G,p_1)$. We prove that this apparently weaker conclusion implies that  
 $(G,p_0)$ and $(G,p_1)$ are not congruent (as long as $G$ has enough vertices) in Section \ref{sec:unique}.

\section{Frameworks on Surfaces}

Let $\M$ be a fixed surface in $\bR^3$.
A framework $(G,p)$ on $\M$ is the combination of a finite graph $G=(V,E)$ and a map $p:V\rightarrow \M$; such a framework is said to be \emph{rigid on $\M$} 
if every continuous motion of the vertices on $\M$ that preserves equivalence also preserves congruence; otherwise $(G,p)$ is said to be \emph{flexible on $\M$}.
Moreover $(G,p)$ is:  \emph{isostatic on $\M$} if it is rigid on $\M$ and, for every edge $e$ of $G$, the framework $(G-e,p)$ is flexible on $\M$;
 \emph{redundantly rigid on $\M$}  if $(G-e,p)$ is rigid on $\M$  for all $e\in E$; and
\emph{globally rigid on $\M$} if every framework $(G,q)$ on $\M$ which is equivalent to $(G,p)$ is congruent to $(G,p)$.

An \textit{infinitesimal flex} $s$ of $(G, p)$ on $\M$ is a map $s:V\to \bR^3$ such that
$s(v)$ is tangential to $\M$ for all $v\in V$ and $(p(u)-p(v))\cdot(s(u)-s(v))=0$
for all $uv\in E$.
The framework $(G,p)$ is \emph{infinitesimally rigid} on $\M$ if every infinitesimal flex of $(G,p)$ is an infinitesimal isometry of $\M$.

We consider four 2-dimensional surfaces:
\begin{itemize}
\item the unit sphere $\S=S^1\times S^1$ centered at the origin, defined by the equation $x^2+y^2+z^2=1$;
\item the unit cylinder $\Y=S^1 \times \bR$ about the $z$-axis, defined by the equation $x^2+y^2=1$;
\item the unit cone $\C$ about the $z$-axis, defined by the  equation $x^2+y^2=z^2$;
\item the ellipsoid $\E$ centered at the origin, defined by the  equation $x^2+{a}y^2+{b}z^2=1$ for some fixed $a,b\in \bQ$ 
with $1<a<b$.
\end{itemize}

These are natural examples of surfaces for which the dimension of the space of infinitesimal isometries is 3, 2, 1 and 0, respectively.
Henceforth, we will use $\M$ to denote one of the four surfaces defined above and $\ell$ for the the dimension of its space of infinitesimal isometries.

The problem of determining whether or not a given framework on $\M$ is rigid is a difficult problem in algebraic geometry. It becomes tractable however if we restrict our attention to `generic' frameworks. We consider a framework $(G,p)$ on $\M$ to be \emph{generic} if $\td[\bQ(p):\bQ]=2|V|$, where $\td[\bQ(p):\bQ]$ denotes the transcendence degree of the field extension. Thus  $(G,p)$ is generic on $\M$ if the coordinates of the vertices of $G$ are as algebraically independent as possible. For generic frameworks the problem of determining rigidity reduces to that of determining infinitesimal rigidity.

\begin{thm}[\cite{NOP}]\label{nop}
Let $(G,p)$ be a generic framework on $\M$.
Then $(G,p)$  is rigid on $\M$ if and only if $G$ is a complete graph on at most $5-\ell$ vertices, or $(G,p)$  is infinitesimally rigid on $\M$.
\end{thm}
We note that \cite{NOP} uses a different definition for a generic framework on $\M$. Corollary \ref{cor:realvar} below verifies that any framework 
which satisfies our definition will also satisfy the definition given in \cite{NOP}.

Theorem \ref{nop}, combined with Theorem \ref{thm:surfacerank} below, imply that the problem of determining generic rigidity on $\M$ depends only on the underlying graph.
This problem has been solved for three of our chosen surfaces.

\begin{thm}[\cite{Lam},\cite{NOP},\cite{NOP2}]\label{thm:cylinderlaman}
Let $(G,p)$ be a generic framework on $\M$ and suppose that $\M\in \{\S,\Y,\C\}$.
Then $(G,p)$ is isostatic on $\M$ if and only if $G$ is $K_n$ for $1\leq n \leq 5-\ell$ or $|E|=2|V|-\ell$ and every subgraph $H=(V',E')$
 of $G$ has $|E'|\leq 2|V'|-\ell$.
\end{thm}

It is an open problem to characterise generic rigidity on the ellipsoid.
The analogous condition to that given in Theorem \ref{thm:cylinderlaman} is known to be necessary:

\begin{lem}[\cite{NOP}]\label{thm:ellipsoidmaxwell}
Let $(G,p)$ be a generic framework on $\E$.
If $(G,p)$ be isostatic then $G$ is $K_n$ for $1\leq n \leq 4$ or $|E|=2|V|$ and every subgraph $H=(V',E')$ of $G$ has $|E'|\leq 2|V'|$.
\end{lem}

However, the graph constructed by adding a vertex of degree two to $K_5$ has an infinitesimal flex for every generic realisation on the ellipsoid. This
shows that the condition  in Lemma \ref{thm:ellipsoidmaxwell} is not sufficient to imply generic rigidity.

\section{Generic Points and Smooth Manifolds}
\label{sec:alggeom}

Let $K,L$ be fields such that $\bQ\subseteq K \subseteq L \subseteq \bC$.
Let $W$ be an algebraic variety over $K$ in $L^n$, i.e. $W=\{x\in L^n: f_i(x)=0 \mbox{ for all } 1 \leq i \leq m\}$ for some $f_1,f_2,\dots, f_m\in K[X]$. We assume that $W$ is \emph{irreducible}; i.e. cannot be expressed as the union of two proper subvarieties.
The {\em dimension} of $W$, $\dim W$, is the maximum length of a chain of subvarieties of $W$.
A point $p\in W$ is {\em generic over $K$} if every $h\in K[X]$ satisfying
$h(p)=0$ has $h(x)=0$ for all $x\in W$. Given an integral domain $R$ we use $\fract(R)$
to denote the field of fractions of $R$.

\begin{lem}\label{lem:genericvariety}
Let $W$ be an irreducible variety over $K$ in $\bC^n$, $p\in W$, and
$I=\{f\in K[X]\,:\,f(x)=0 \mbox{ for all }x\in W\}.$
Then:\\
(a) $\dim W=\td[\fract(K[X]/I):K]$;\\
(b) The map $h+I\mapsto h(p)$ is a surjective ring homomorphism from $K[X]/I$
to $K(p)$, and is a ring isomorphism if and only if $p$ is a generic point of $W$ over $K$;\\
(c) $\td [K(p):K]\leq \dim W$, and equality holds if and only if $p$ is a generic point of $W$ over $K$.
\end{lem}

\begin{proof} Part (a) is well known (for example see \cite[Exercise 3.20 (b)]{Har}) and Part (b) is elementary.
Part (a) along with the first part of (b) implies that
$$\dim W=\td[\fract(K[X]/I):K]\geq \td[K(p):K].$$
The second part of (b) tells us that if $p$ is generic then equality
holds in the above inequality. It remains to show that $p$ is a generic point of $W$ over $K$
when $\td [K(p):K]= \dim W$. Let $h\in K[X]$ with $h(p)=0$,
$W_1=\{x\in W\,:\,h(x)=0\}$, and let $W_2$ be the irreducible component of $W_1$ which contains
$p$.  The argument in the second sentence of the proof tells
us that $\dim W_2\geq \td[K(p):K]=\dim W$. The definition of dimension and the fact that
$W_2$ is a subvariety of $W$ now imply that $W_2=W$. Hence $h(x)=0$ for all $x\in W$ and $p$ is a generic point of $W$ over $K$.
\end{proof}

\begin{cor}\label{cor:realvar}
Let $K$ be a field with $\bQ\subseteq K\subseteq \bR$, $W$ be an irreducible variety over $K$ in $\bC^n$ and $p\in W\cap \bR^n$. If $\td[K(p):K]=\dim W$ then $p$ is a generic point of $W$ and hence is also a generic point of $W\cap \bR^n$.
\end{cor}
It follows that if $(G,p)$ is a generic framework with $n$ vertices on $\M$, then $p$ is a generic point of the irreducible variety $\M^n$ over $K$.

\medskip

Let $X$ be a smooth manifold and $f:X\rightarrow \bR^m$ be a smooth map. Then $x \in X$ is said to be a \emph{regular point of $f$} if $df|_x$ has maximum rank and is a \emph{critical point of $f$} otherwise. Also $f(x)$ is said to be a \emph{regular value of $f$} if, for all $y \in f^{-1}(f(x))$, $y$ is a regular point of $f$; otherwise $f(x)$ is a \emph{critical value of $f$}.

\begin{lem}[\cite{J&K}]\label{lem:JK2011}
Let $M$ be a smooth manifold and $x\in M$. Suppose $\theta:M\rightarrow \bR^a$ and $F:M\rightarrow \bR^b$ are smooth maps. Define $H:M\rightarrow \bR^{a+b}$ by $H(y)=(F(y),\theta(y))$. Suppose $\theta(x)$ is a regular value of $\theta$. Let $X$ be the submanifold $\theta^{-1}(\theta(x))$ of $M$, and let $f$ be the restriction of $F$ to $X$. Then $\rank df|_x=\rank dH|_x - \rank d\theta|_x$.
\end{lem}

Note that $\theta^{-1}(\theta(x))$ is a submanifold of $M$ by the following well known result, see for example \cite[Page 11, Lemma 1]{Mil}.

\begin{lem}\label{lem:milnorp11lem1}
Let $X$ and $Y$ be smooth manifolds
 and $f:X\rightarrow Y$ be a smooth map. Suppose that $M$ has dimension $m$, $x \in X$, $f(x)$ is a regular value of $f$ and $\rank df|_x=t$. Then $f^{-1}(f(x))$ is an $(m-t)$-dimensional smooth manifold.
\end{lem}

\section{The Rigidity Map}
\label{sec:gen}

We assume henceforth that $G=(V,E)$ is a graph with $V=\{v_1,v_2,\ldots,v_n\}$ and $E=\{e_1,e_2,\ldots,e_m\}$.
The {\em rigidity map} $f_G:\bR^{3n}\rightarrow \bR^m$ is defined by $f_G(p)=(\|e_1 \|^2, \dots, \|e_m\|^2)$ where $\|e_i\|^2=\|p(v_j)-p(v_k)\|^2$
when $e_i=v_jv_k$.
Let $\theta_G:\bR^{3n}\rightarrow \bR^n$ be the map defined by $\theta_G(p)=(h(p(v_1)),\dots, h(p(v_n)))$ where: 
$$h(x_i,y_i,z_i)= \begin{cases}
x_i^2+y_i^2+z_i^2-1, & \text{if } \M=\S; \\
 x_i^2+y_i^2-1, & \text{if }\M=\Y; \\
 x_i^2+y_i^2-z_i^2, & \text{if }\M=\C; \\
  x_i^2+{a}y_i^2+{b}z_i^2-1, & \text{if }\M=\E.
\end{cases} $$
The
{\em $\M$-rigidity map} $F_G:\bR^{3n}\rightarrow \bR^{m+n}$ is defined by $F_G=(f_G,\theta_G)$.
The Jacobian matrix for the derivative of $F_G$ evaluated at any point $p\in \M^n$ is (up to scaling) the {\em rigidity matrix} for the framework $(G,p)$ on $\M$.
It is shown in \cite{NOP} that the null space of this matrix is the space of infinitesimal flexes of $(G,p)$ on $\M$. This allows us to
characterise isostatic generic frameworks in terms of $dF_G$.

\begin{thm}[\cite{NOP}]\label{thm:surfacerank}
Let $(G,p)$ be a generic framework on $\M$. Then $(G,p)$ is
isostatic on $\M$ if and only if $ m =2n-\ell$ and $\rank dF_G|_p = 3n-\ell$.
\end{thm}

Note that for any $p=(x_1,y_1,z_1,\ldots,x_n,y_n,z_n)\in \M^n$ we have
\[ d\theta_G|_p = \begin{bmatrix} dh(p(v_1)) & 0 & \dots & 0 \\  0 & dh(p(v_2)) & \dots & 0 \\ \vdots && \ddots & \vdots\\ 0 & 0 & \dots & dh(p(v_n)) \end{bmatrix} \]
where:
$$dh(x_i,y_i,z_i)= \begin{cases}
 2(x_i,y_i,z_i), & \text{if }\M=\S; \\
 2(x_i,y_i,0), & \text{if }\M=\Y; \\
 2(x_i,y_i,-z_i), & \text{if }\M=\C; \\
  2(x_i,ay_i,{b}z_i), & \text{if }\M=\E.
\end{cases} $$
It follows that $\rank d\theta_G|_p=n$ unless $p(v_i)=(0,0,0)$ for some $1\leq i\leq n$ (which can only occur if $\M=\C$).
Thus $p$ is a regular point of $\theta_G$ for all such $p$. Moreover we have $\theta_G(p)=0$ is a regular value of $\theta_G$ for all $p \in \M^n$ when
$\M\in \{\S,\Y,\E\}$ (and Lemma \ref{lem:milnorp11lem1} applied to $\theta_G$ at any $p\in \M^n$ confirms that $\M^n$ is a manifold of dimension $3n-n=2n$
in these three cases).

\medskip

We say that two frameworks $(G,p)$ and $(G,q)$ on $\M$ are {\em $\M$-congruent} if there is an isometry of $\M$ which maps $(G,p)$ to $(G,q)$.
The framework $(G,p)$ is  in \emph{standard position on $\M$} (with respect to the fixed ordering of the vertices) if $p(v_1)=(x_1,y_1,z_1)$, $p(v_2)=(x_2,y_2,z_2)$ and: $(x_1,y_1,z_1)=(0,1,0)$ and $x_2=0$ if $\M=\S$; $(x_1,y_1,z_1)=(0,1,0)$ if $\M=\Y$; and $x_1=0$ if $\M=\C$. When $\M=\E$ all frameworks  are considered to be in standard position. It is easy to see that any framework $(G,p)$ on $\M$ is $\M$-congruent to a framework
in standard position on $\M$ (as long as $p(v_i)\neq (0,0,0)$ for all $1\leq i\leq n$ when $\M=\C$).
We consider frameworks $(G,p)$ in standard position in order to factor out the continuous isometries of $\M$.

We will need a version of the $\M$-rigidity map for frameworks which are constrained to lie in standard position on $\M$.
Define $\alpha:\bR^{3n}\to \bR^\ell$ by:
$$\alpha(x_1,y_1,z_1,\ldots,x_n,y_n,z_n)= \begin{cases}
 (x_1,z_1,x_2), & \text{if }\M=\S; \\
 (x_1,z_1), & \text{if }\M=\Y;\\
 x_1, & \text{if }\M=\C.
\end{cases} $$
Let $\theta^*_G:\bR^{3n}\to \bR^{n+\ell}$ by $\theta^*_G=(\theta_G,\alpha)$ if $\M\in \{\S,\Y,\C\}$ and
$\theta^*_G=\theta_G$ if $\M=\E$. Let
$F^*_G:\bR^{3n}\to \bR^{m+n+\ell}$ by $F^*_G=(f_G,\theta_G^*)$.
Then the null space of $dF^*_G|_p$ is the space of all infinitesimal flexes of $(G,p)$ which leave
the coordinates in $\alpha(p)$ fixed.

\smallskip

We say a framework $(G,p)$ on $\M$ is \emph{independent} if $\rank dF_G|_p = m+n$ i.e. the rows of (the Jacobian matrix for)
$dF_G|_p$ are linearly independent. 

\begin{lem}\label{lem:matrix}
Let $(G,p)$ be an independent framework on $\M$. Then $dF^*_{G}|_p$ has linearly independent rows and hence 
$\rank dF^*_{G}|_p=m+n+\ell$.
\end{lem}
\begin{proof} 
We first consider the case when $(G,p)$ is isostatic, and hence infinitesimally rigid. Since 
 the null space of $dF^*_G|_p$ is the space of all infinitesimal flexes of $(G,p)$ which leave
the coordinates in $\alpha(p)$ fixed, this space is trivial. Hence $\rank dF^*_G|_p=3n=n+m+\ell$ and the rows of $dF^*_{G}|_p$ are 
linearly independent. The case when $(G,p)$ is not isostatic follows
since any independent framework $(G,p)$ can be extended to an isostatic framework $(H,p)$, and $dF^*_G|_p$ can then be obtained by deleting rows from
$dF^*_{H}|_p$.
\end{proof}

\section{Quasi-Generic Frameworks}
\label{sec:quasigen}

Two frameworks $(G,p)$ and $(G,q)$ on $\M$ are said to be  {\em $\M$-congruent} if there is an isometry of $\M$ which maps $(G,p)$ to $(G,q)$. Note that
this is a stronger condition than congruence: we may have two equivalent realisations of $K_3$ on the cylinder which are not $\M$-congruent (but clearly must be congruent).
A framework $(G,p)$ is \emph{quasi-generic} on $\M$ if it is $\M$-congruent to a generic framework. 

Let
$$\N=\{p\in \M^n: (G,p) \text{ is in}\break \text{ standard position on $\M$ and $p(v)\neq (0,0,0)$ for all $v\in V$}\}.$$
(Note that the condition that $p(v)\neq (0,0,0)$ for all $v\in V$ is only relevant when $\M=\C$.)
Then $\N$ is a $(2n-\ell)$-dimensional submanifold of $\M^n$.
Let $f_{[G]}=f_{G}|_{\N}$.

\begin{lem}\label{lem:regpoint}
Let $(G,p)$ be an independent framework in standard position on $\M$. Then $\rank df_{[G]}|_p=m$ and hence $p$ is a regular point of $f_{[G]}$.
\end{lem}

\begin{proof}
Consider the map $F_G^*=(f_G,\theta^*_G)$ defined in Section \ref{sec:gen}.

We first consider the case when $\M\in \{\S,\Y,\E\}$. We can use a similar argument to that given after Theorem \ref{thm:surfacerank} to show that
$\rank d\theta_G^*|_p=n+\ell$ and hence $\theta_G^*(p)$ is a regular value of $\theta_G^*$.
Since $\N=(\theta_G^*)^{-1}(\theta_G^*(p))$, we may now use Lemma \ref{lem:JK2011} to deduce that,
$$\rank df_{[G]}|_p = \rank dF^*_{G}|_p - \rank d\theta_G^*|_p= m+n+\ell - (n+\ell)=m.$$

We proceed similarly in the case when $\M=\C$, but we have to restrict the domain of $F_G^*$ to an open neighbourhood 
of $p$ which contains no critical points of $F_G^*$, i.e. contains no points $p$ with $p(v)=(0,0,0)$ for some $v\in V$,
otherwise $p$ is no longer a regular value of $\theta_G^*$. (We have $\theta_G^*(p)=(0,0,\ldots,0,x_1)=\theta_G^*(q)$ for $q=(x_1,0,x_1,0,0,\ldots,0)$.)
\end{proof}

\begin{lem}\label{lem:gen->edgegen}
Let $(G,p)$ be an isostatic quasi-generic framework on $\M$.
Then $$\td[\bQ(f_G(p)):\bQ]=m.$$
\end{lem}

\begin{proof}
Without loss of generality we may assume $(G,p)$ is in standard position.
Let $f_{[G]}(p)=(\beta_1,\dots,\beta_{m})=\beta$. Suppose $g(\beta)=0$ for some polynomial $g$ with rational coefficients. Then $g\circ f_{[G]}(p)=0$.
Since $(G,p)$ is quasi-generic and $g\circ f_{[G]}$ is a polynomial with rational coefficients, Corollary \ref{cor:realvar} implies that $g\circ f_{[G]}(q)=0$ for all $q\in \M^{n}$. In particular $g\circ f_{[G]}(q)=0$ for all $q\in \N$.
By Lemma \ref{lem:regpoint}, $\rank df_{[G]}|_q=m$.
Since $(G,p)$ is isostatic, $\N$ is $m$-dimensional and hence the Inverse Function Theorem implies that $f_{[G]}$ maps some open neighbourhood $U$ of $p$ in $\N$ diffeomorphically onto some neighbourhood $f_{[G]}(U)$ of $f_{[G]}(p)$ in $\bR^m$. Then $g(y)=0$ for all $y$ in the open subset $f_{[G]}(U)$
of $\bR^m$. This implies that $g$ must be the zero polynomial and hence
 $\{\beta_1,\dots,\beta_{m}\}$ is algebraically independent over $\bQ$.
\end{proof}

\begin{lem}\label{lem:q-g<=>}
A framework $(G,p)$ on $\M$ is quasi-generic if and only if $(G,p)$ is $\M$-congruent to a framework $(G,q)$ in standard position with $\td[\bQ(q):\bQ]=2n-\ell$.
\end{lem}

\begin{proof}
Suppose that $(G,p)$ is quasi-generic. Then $(G,p)$ is $\M$-congruent to a framework $(G,q)$ in standard position.
Since the definition of quasi-generic depends only on $p$, we may assume, without loss of generality, that
$G$ is isostatic. By Lemma \ref{lem:gen->edgegen}, $\td[\bQ(f_{G}(p)):\bQ]=\td[\bQ(f_{G}(q)):\bQ]=2n-\ell$.
Let $K$ and $L$ be the algebraic closures of $\bQ(f_{G}(q))$ and $\bQ(q)$ respectively. Since $K\subseteq L$ we have
$$\td[\bQ(q):\bQ]=\td[L:\bQ]\geq \td[K:\bQ]= 2n-\ell.$$
Since $(G,q)$ is in standard position on $\M$, we also have $\td[\bQ(q):\bQ]\leq 2n-\ell$. Hence $\td[\bQ(q):\bQ]= 2n-\ell$.

Now suppose $(G,p)$ is $\M$-congruent to a framework $(G,q)$ in standard position with $\td[\bQ(q):\bQ]=2n-\ell$. Let $q=(q(v_1),x_2,y_2,z_2,\dots,x_n,y_n,z_n)$.
When $\M=\E$, we have $\ell=0$ and $(G,q)$ is generic so we are done.

Suppose $\M=\Y$. As $\td[\bR:\bQ]=\infty$ we may choose $\theta$ such that $\{\sin \theta,x_2,\dots,x_n,\break z_2,\dots,z_n\}$ is algebraically independent over $\bQ$. Apply the rotation
\[ \begin{bmatrix} \cos \theta & -\sin \theta & 0\\ \sin \theta & \cos \theta & 0\\ 0 & 0 & 1 \end{bmatrix} \]
to each of the vertices of $(G,q)$ to achieve the equivalent framework $(G,q')$. Then
$$q'=(-\sin \theta, \cos \theta,0,x_2\cos \theta-y_2\sin \theta,x_2\sin\theta+y_2\cos\theta,z_2,x_3\cos \theta-y_3\sin \theta,$$$$x_3\sin\theta+y_3\cos\theta,z_3,\newline
\dots,x_n\cos \theta-y_n\sin \theta,x_n\sin\theta+y_n\cos\theta,z_n).$$
As $\{\sin\theta,x_2,\dots,x_n,z_2,\dots,z_n\}$ is algebraically independent over $\bQ$, $\M=\Y$ and $x_i^2+y_i^2=1$ for $2\leq i \leq n$, $\{\sin \theta,x_2\cos\theta-y_2\sin\theta,\dots,x_n\cos\theta-y_n\sin\theta,z_1,\dots,z_n\}$ is algebraically independent over $\bQ$.
Next choose $t\in \bR$ such that $S=\{\sin\theta,x_2\cos \theta-y_2\sin \theta,\dots,x_n\cos \theta-y_n\sin \theta,z_1+t,\dots,z_n+t\}$ is algebraically independent over $\bQ$. Translating $(G,q')$ parallel to the $z$-axis by $t$ we reach the equivalent framework $(G,q'')$ where $q''=(-\sin\theta,\cos\theta, t, x_2\cos \theta-y_2\sin \theta,x_2\sin\theta+y_2\cos\theta,z_2+t,\dots,x_n\cos \theta-y_n\sin \theta,x_n\sin\theta+y_n\cos\theta,z_n+t)$. As $S$ is algebraically independent over $\bQ$ we have that $\td[\bQ(q''):\bQ]=2n$. Thus $q''$ is generic on $\Y$. Since $(G,p)$ is $\M$-congruent to $(G,q'')$, $(G,p)$ is quasi-generic.

The remaining cases, when $\M=\S$ or $\M=\C$, follow by a similar argument.
\end{proof}

\begin{lem}\label{lem:quasi=edgegen}
Let $(G,p)$ be an isostatic framework on $\M$. Then $(G,p)$ is quasi-generic if and only if $\td[\bQ(f_G(p)):\bQ]=2n-\ell$.
\end{lem}

\begin{proof}
We may assume that $(G,p)$ is in standard position on $\M$.
If $(G,p)$ is quasi-generic, then $\td[\bQ(f_G(p)):\bQ]=2n-\ell$ by Lemma \ref{lem:gen->edgegen}.

Suppose on the other hand that $\td[\bQ(f_G(p)):\bQ]=2n-\ell$.
Since $(G,p)$ is in standard position on $\M$ we have $\td[\bQ(p):\bQ]\leq 2n-\ell$.
Since $\bQ(f_G(p))\subseteq\bQ(p)$ and $\td[\bQ(f_G(p)):\bQ]=2n-\ell$, we must have $\td[\bQ(p):\bQ]= 2n-\ell$.
Lemma \ref{lem:q-g<=>} now tells us that $(G,p)$ is quasi-generic.
\end{proof}

\section{Regular Maps}
\label{sec:regular}

Suppose $(G,p)$ is in standard position on $\M$.
Lemma \ref{lem:regpoint} implies that $p$ is a regular point of $f_{[G]}$ when $(G,p)$ is independent.
We will use the following fundamental theorems to prove that $f_{[G]}(p)$ is a regular value of $f_{[G]}$ when $(G,p)$ is quasi-generic.

Recall that a subset $S$ of $\bR^n$ is {\em semi-algebraic over $\bQ$} if it
can be expressed as a finite union of sets of the form
$$\{x\in \bR^n : \mbox{$P_i(x)=0$ for $1\leq i\leq s$  and $Q_j(x)>0$ for $1\leq j\leq t$} \},$$
where $P_i\in \bQ[X_1,\ldots,X_n]$ for $1\leq i\leq s$, and
$Q_j\in \bQ[X_1,\ldots,X_n]$ for $1\leq j\leq t$.

\begin{thm}[Tarski-Seidenberg \cite{Sei}, \cite{Tar}]\label{thm:T-S}
Let $S\subset \bR^{n+k}$ be semi-algebraic over $\bQ$ and $\pi:\bR^{n+k}\rightarrow \bR^n$ be the projection onto the first $n$ coordinates. Then $\pi(S)$ is semi-algebraic over $\bQ$.
\end{thm}

\begin{thm}[Sard, see {\cite[Page 16]{Mil}}]\label{thm:sard}
Let $f:U\rightarrow \bR^n$ be a smooth map defined on an open subset $U$ of $\bR^m$ and $C$ be the set of critical points of $f$. Then $f(C)$ has Lebesque measure zero in $\bR^n$.
\end{thm}

We also need the following elementary result, see for example \cite[Lemma 3.3]{J&K}.

\begin{lem}\label{lem:JK3.3}
Let $M$ be a smooth manifold and let $f:M\rightarrow \bR^n$ be a smooth map. Let $x\in M$ and choose an open neighbourhood $U$ of $x$ on $M$ such that $U$ is diffeomorphic to $\bR^m$.
Let $g$ be the restriction of $f$ to $U$ and let $x$ be a regular point of $g$. Suppose that the rank of $dg|_x$ is $n$. Then there exists an open neighbourhood $W \subseteq U$ of $x$ such that $g(W)$ is an open neighbourhood of $g(x)$ in $\bR^n$.
\end{lem}

\begin{lem}\label{lem:regval}
Let $(G,p_0)$ be a quasi-generic framework in standard position on $\M$. Then $f_{[G]}(p_0)$ is a regular value of $f_{[G]}$.
\end{lem}

\begin{proof}
We first consider the case when $(G,p_0)$ is independent.
Let $t=\max\{\rank df_{[G]}|_p: p\in \N\}$. Let
$$S=\{(p,q)\in \N\times \N: f_{[G]}(p)=f_{[G]}(q)\text{ and } \rank df_{[G]}|_q<t\}$$
and
$$S'=\{p\in \N:\text{ there exists }q\in \N\text{ such that }(p,q)\in S\}.$$
The set $S$ is an algebraic set, so Theorem \ref{thm:T-S} implies that $S'$ is a semi-algebraic set.

Suppose that $p_0\in S'$. Then there exists a set
$$S''=\{x\in \N:P_i(x)=0\text{ and }Q_j(x)>0\text{ for }1\leq i\leq a \text{ and }1\leq j\leq b\}\subseteq S$$
that contains $p_0$ where $P_i,Q_j \in \bQ[x_1,y_1,z_1, x_2,y_2,z_2,\dots, x_n,y_n,z_n]$ and $a,b \geq 0$. Let $I=\{g \in \bQ[x_1,y_1,z_1, x_2,y_2,z_2,\dots, x_n,y_n,z_n]:g(x)=0 \mbox{ for all } x \in \N\}$, then $P_i \in I$ for all $1\leq i \leq a$.

Since $p_0$ is generic on $\N$ we have
$$S''=\{x\in \N:Q_j(x)>0\text{ for }1\leq j\leq b\}.$$
Therefore there exists a neighbourhood $W\subseteq \N$ of $p_0$ such that $W\subseteq S''$. By Lemmas \ref{lem:regpoint} and \ref{lem:JK3.3} we can choose $W'\subset W$ such that $f_{[G]}(W')$ is an open subset of $\bR^{m}$. Then each point of $f_{[G]}(W')$ is a critical value of $f_{[G]}$, contradicting Theorem \ref{thm:sard}.
Thus $p_0\not\in S'$ and hence $f_{[G]}(p_0)$ is a regular value of $f_{[G]}$.

When $(G,p_0)$ is not independent we choose a maximal subgraph $H$ such that $(H,p_0)$ is independent. Then $f_{[H]}(p_0)$ is a regular value of $f_{[H]}$ and hence $f_{[G]}(p_0)$ is a regular value of $f_{[G]}$.
\end{proof}

\begin{lem}\label{lem:smoothmanifold}
Let $(G,p)$ be a quasi-generic framework in standard position on $\M$ and $r=\rank df_{[G]}(p)$. Then $f_{[G]}^{-1}(f_{[G]}(p))$ is a $(2n-\ell-r)$-dimensional smooth manifold.
\end{lem}

\begin{proof}
The domain of $f_{[G]}$ is $\N$ which is a $(2n-\ell)$-dimensional manifold.
Lemma \ref{lem:regval} shows that $f_{[G]}(p)$ is a regular value of $f_{[G]}$.
The result now follows by applying Lemma \ref{lem:milnorp11lem1}.
\end{proof}

\section{Unique surfaces containing a given set of points}
\label{sec:unique}

We show in this section that if $G$ has sufficiently many vertices and $(G,p)$ and $(G,q)$ are congruent realisations of $G$ on $\M$, then 
they are $\M$-congruent; i.e. there is an isometry of $\M$ which maps $(G,p)$ onto $(G,q)$. 

We say that two surfaces in $\bR^3$ are {\em congruent} if there is an isometry of $\bR^3$ mapping one to the other. We first show that if we choose a set $S$ of sufficiently many generic points on $\M$, then the only surface which is congruent to $\M$ and contains $S$ is $\M$ itself.

\begin{lem}\label{lem:uniquec}
Let $(K_{4+\gamma},p)$ be generic on $\M$ where $\gamma=0$ if $\M=\S$, $\gamma=1$ if $\M=\Y$ and $\gamma = 2$ otherwise. Let $\T$ be another surface in $\bR^3$ which is congruent to $\M$ and contains $(K_{4+\gamma},p)$. Then $\T=\M$.
\end{lem}

\begin{proof}
First suppose $\M=\Y$ and hence $\gamma=1$. Then $\T$ is a congruent cylinder containing $(K_{5},p)$. Assume the axis of $\T$ has a non-zero $z$-component and let $(c_1,c_2,0)$ be the point where it crosses the $xy$-plane. Let $(c_3,c_4,1)$ be a direction vector for the axis. Let $\bK=\bQ(c_1,c_2,c_3,c_4)$. Let $Z_1,Z_2$ be the complex extensions of $\Y$ and $\T$ respectively, i.e. the sets of all complex solutions to the equations which define $\Y$ and $\T$ respectively. Let $M_i$ be the irreducible component of $Z_1\cap Z_2$ which contains $p(v_i)$ for $1\leq i \leq 5$. Each $M_i$ is a proper subvariety of $\Y$ so has dimension at most 1. Let $M$ be the direct product of $M_1,\dots, M_5$. Then $M$ is irreducible and $\dim M\leq 5$.
By Lemma \ref{lem:genericvariety} (c) it follows that $\td[\bK(p):\bK]\leq 5$. Thus $\td[\bK(p):\bQ]\leq 9$. However, since $p$ is generic and $\dim \Y^5=10$, we have $\td[\bQ(p):\bQ]=10$, a contradiction.
The case when the axis of $\T$ is contained within the $xy$-plane follows similarly by choosing the direction vector $(c_3,c_4,0)$ for the axis.

For the other cases, we use the following parametrisations and apply a similar argument. For $\M=\S$, choose the parametrisation of $\T$ with centre $(c_1,c_2,c_3)$. For $\M=\C$, choose the parametrisation of $\T$ with centre $(c_1,c_2,c_3)$ and a direction vector for the axis $(c_4,c_5,1)$, unless the axis is contained in the plane $z=0$, in which case take the axis $(c_4,c_5,0)$. For $\M=\E$, choose the parametrisation of $\T$ with the given radii having centre $(c_1,c_2,c_3)$ and a direction $(c_4,c_5,1)$ for the longest radius.
\end{proof}

\begin{lem}\label{lem:defnglobal}
Let $n\geq 4+\gamma$ where $\gamma=0$ if $\M=\S$, $\gamma=1$ if $\M=\Y$ and $\gamma = 2$ otherwise. Let $(K_n,p)$ and $(K_n,q)$ be equivalent frameworks on $\M$. Then there is an isometry $\iota$ of $\M$ such that $\iota(p)=q$.
\end{lem}

\begin{proof}
It suffices to prove the case $n=4+\gamma$. Since $(K_{4+\gamma},p)$ and $(K_{4+\gamma},q)$ are congruent there is an isometry $\iota$ of $\bR^3$ which maps $(K_{4+\gamma},p)$ onto $(K_{4+\gamma},q)$. Then $\iota(\M)$ is a surface in $\bR^3$ which is congruent to $\M$. By Lemma \ref{lem:uniquec}, $\M$ is the unique surface of $\bR^3$ which contains  $(K_{4+\gamma},q)$. Hence $\iota(\M)=\M$ and $\iota$ is an isometry of $\M$.
\end{proof}

Note that, in the case where $\M=\E$, there are no continuous isometries so $\iota$ must be a discrete isometry.

\section{Global Rigidity}\label{sec:global}

We show that the known necessary conditions for global rigidity in $\bR^d$ given in Theorem \ref{thm:hennecessary} have natural analogues for generic frameworks on surfaces.
First we consider $k$-connectivity.

\begin{prop}\label{prop:kcon}
Let $G=(V,E)$ with $n\geq 4$. Let $(G,p)$ be a generic globally rigid framework on $\M$. Then $G$ is $k$-connected where $k=3$ if $\M=\S$, $k=2$ if $\M\in \{\Y,\C\}$ and $k=1$ if $\M=\E$.
\end{prop}

\begin{proof}
Assume $G$ is not $k$-connected. Then we have $G=G_1\cup G_2$ for subgraphs $G_i=(V_i,E_i)$ with $|V_1\cap V_2|\leq k-1$.
Let $p_1=p|_{V_1}$ and $p_2=p|_{V_2}$. Let $(G,q)$ be obtained from $(G,p)$ by reflecting $(G_2,p_2)$ in a plane which contains
$p(V_1\cap V_2)$ and also contains:
the origin when $\M=\S$;
the $z$-axis when $\M\in\{\Y,\C\}$;
the $y$- and $z$-axes when $\M=\E$.
Then $(G,q)$ is a framework on $\M$ and is equivalent but not congruent to $(G,p)$.
\end{proof}

We next consider redundant rigidity.

\begin{thm}\label{thm:redundant}
Suppose $(G,p_0)$ is quasi-generic and globally rigid on $\M$ and $n\geq 4+\gamma$ where $\gamma=0$ if $\M=\S$, $\gamma=1$ if $\M=\Y$ and $\gamma=2$ otherwise. Then $(G,p_0)$ is redundantly rigid on $\M$.
\end{thm}

\begin{proof}
Without loss of generality we may assume that $(G,p_0)$ is in standard position.
Since $(G,p_0)$ is globally rigid, it is rigid. Suppose, for a contradiction, that $(G,p_0)$ is not redundantly rigid.
Let $G=(V,E)$ and choose $e\in E$ such that $(G-e,p_0)$ is not rigid.
Since $n\geq 4+\gamma$, we have $\rank df_{[G]}(p)=2n-\ell$ and $\rank df_{[G-e]}(p)=2n-\ell-1$. Hence, by Lemma \ref{lem:smoothmanifold}, $\P=f_{[G-e]}^{-1}(f_{[G-e]}(p_0))$ is a one dimensional submanifold of $\N$.

\begin{claim}\label{c1}
$\P$ is a closed and bounded subset of $\bR^{3n}$.
\end{claim}
\begin{proof}
We first verify closure. Let $p_1,p_2,\ldots$ be a sequence of points in $\P$ which converge to a limit point $p\in \bR^{3n}$.
Then $(G-e,p_1),(G-e,p_2),\ldots$ is a sequence of frameworks in standard position on $\M$ which are equivalent to $(G-e,p_0)$ and converge to a limit
$(G-e,p)$. It is easy to see that $(G-e,p)$ will be in standard position on $\M$ and equivalent to $(G-e,p_0)$. This is enough to imply that $p\in \P$ when
$\M\neq \C$. When $\M= \C$ we must also verify that $p(v)\neq (0,0,0)$ for all $v\in V$. In this case we have $\td[\bQ(f_G(p_0)):\bQ]\geq 2n-1$ by Lemma \ref{lem:gen->edgegen}. Since $f_{G-e}(p)=f_{G-e}(p_0)$ and  $\bQ(f_{G-e}(p))\subseteq \bQ(p)$, this gives
$\td[\bQ(p):\bQ]\geq 2n-2$. If $p(v)= (0,0,0)$ for some $v\in V\setminus \{v_1\}$, then we would also have $\td[\bQ(p):\bQ]\leq 2n-3$, which is impossible.

We next verify boundedness. When $\M=\E$, this follows from the fact $\P\subset \E^n$ and $\E^n$ is bounded. When $\M\neq \E$, it follows from the facts that  $G-e$ is connected by Proposition \ref{prop:kcon} and that $(G-e,q)$ is in standard position on $\M$ and equivalent to $(G-e,p_0)$ for all
$q\in \P$.
\end{proof}

Let  $\O$ be the component of $\P$ which contains $p_0$. Claim \ref{c1} and the classification of one dimensional manifolds tells us
that $\O$ is diffeomorphic to a circle.
For any $q\in \N$, we consider the framework $(G,q)$ and define the map $f_e:\N\rightarrow\bR$ by
$f_e(q)=\|q(u)-q(v)\|^2$, where $e=uv$. 
Then $f_{[G]}=(f_e,f_{[G-e]})$. Let $f_{[e]}=f_e|_\P$.

By Lemma \ref{lem:regval}, $f_{[G-e]}(p_0)$ is a regular value of $f_{[G-e]}$. Since $\P=f_{[G-e]}^{-1}(f_{[G-e]}(p_0))$, Lemma \ref{lem:JK2011} gives
$$\rank df_{[e]}|_{p_0}=\rank df_{[G]}|_{p_0}-\rank df_{[G-e]}|_{p_0}=2n-\ell-(2n-(\ell+1))=1.$$
Hence $p_0$ is not a critical point of $f_{[e]}$, and there exists $q_1,q_2\in \O$ with $f_e(q_1)<f_e(p_0)<f_e(q_2)$. There are two paths in $\O$ between $q_1$ and $q_2$, so by the Intermediate Value Theorem there exists a $p_1\in \O$ with $p_1\neq p_0$ and $f_e(p_1)=f_e(p_0)$.  Then $(G,p_o)$ is equivalent to $(G,p_1)$. 
We may assign an orientation to $\O$ and suppose that $p_1$ has been chosen such that  $p_1$ is as close to $p_0$ as possible when we traverse $\O$ in the forward direction.
If $(G,p_0)$ is not congruent to $(G,p_1)$ we are done so we may assume that $(G,p_0)$ is congruent to $(G,p_1)$. Then Lemma \ref{lem:defnglobal} implies there is an isometry $\iota$ of $\M$ such that $\iota(p_0)=p_1$. Since $(G,p_0)$ and $(G,p_1)$ are both in standard position with respect to $v_1$, $\iota$ is a discrete isometry of $\M$; i.e. $\iota$ is a composition of reflections of $\M$ in some planes.
Given any point $p\in \O$, let $p^{-1}=\iota(p)$. Given any path $\alpha:[0,1]\rightarrow \O$ in $\O$ from $p_0$ to $p_1$, 
define the path $\alpha^{-1}:[0,1]\rightarrow \O$ by $\alpha^{-1}(t)=\alpha(t)^{-1}$. Then $\alpha^{-1}$ is a path in $\O$ from $p_1$ to $p_0$.

First suppose that $\alpha$ and $\alpha^{-1}$ have different images in $\O$. Then $\alpha$ and $\alpha^{-1}$ cover $\O$. Without loss of generality suppose that $f_e$ increases as we pass through $p_0$ in the forward direction. Then $f_e$ also increases as we pass through $p_1$ in the forward direction. Hence there exists $t_1,t_2$ with $0<t_1<t_2<1$ such that $f_e(p_{t_1})>f_e(p_0)$ and $f_e(p_{t_2})< f_e(p_1)=f_e(p_0)$. By the Intermediate Value Theorem, there exists $t_3\in [t_1,t_2]$ with $f_e(p_{t_3})=f_e(p_0)$. Then $(G,p_{t_3})$ is equivalent to $(G,p_0)$ and $p_{t_3}$ contradicts the choice of $p_1$.

Now suppose $\alpha$ and $\alpha^{-1}$ have the same image in $\O$. Then $\alpha$ and $\alpha^{-1}$ traverse the same segment of $\O$ in opposite directions. Call the direction from $p_0$ to $p_1$ forward. By the Intermediate Value Theorem there exists $t\in [0,1]$ such that $\alpha(t)=\alpha^{-1}(t)$. Putting $\alpha(t)=p_t$ we have $p_t^{-1}=p_t$. We will show that this contradicts the fact that $(G,p_0)$ is quasi-generic.
Consider the following cases.

\smallskip
\noindent
{\bf Case 1}: $\M=\S$. Then $\iota$ is the unique reflection in the plane $x=0$. Thus, for any realisation $(G,p)$, if $p(v_i)=(x_i,y_i,z_i)$ we have $p^{-1}(v_i)=(-x_i,y_i,z_i)$. Since $p_t(v_i)=p_t^{-1}(v_i)$ we have $p_t(v_i)=(0,y_i,z_i)$ for all $v_i \in V$. Since $p_t(v_1)=(0,1,0)$ and $x_i^2+y_i^2+z_i^2=1$ for all $2\leq i \leq n$ this gives $\td[\bQ(p_t):\bQ]\leq n-1$. But $f_{G-e}(p_t)=f_{G-e}(p_0)$ and $\td[\bQ(f_G(p_0)):\bQ]=2n-3$, hence $\td[\bQ(f_{[G-e]}(p_t)):\bQ]=2n-4$, which gives a contradiction.

\smallskip
\noindent
{\bf Case 2}: $\M=\Y$. Then $\iota$ is a composition of reflections in the plane $z=0$ and the plane through $(0,1,0)$ and the $z$-axis.

We first consider the subcase where $\iota$ is the reflection in the $z=0$ plane. Then, for any realisation $(G,p)$, if $p(v_i)=(x_i,y_i,z_i)$ we have $p^{-1}(v_i)=(x_i,y_i,-z_i)$. Since $p_t(v_i)=p_t^{-1}(v_i)$ we have $p_t(v_i)=(x_i,y_i,0)$ for all $v_i \in V$. Since $p_t(v_1)=(0,1,0)$ and $x_i^2+y_i^2=1$ for all $2\leq i \leq n$ this gives $\td[\bQ(p_t):\bQ]\leq n-1$. But $f_{G-e}(p_t)=f_{G-e}(p_0)$ and $\td[\bQ(f_G(p_0)):\bQ]=2n-2$, hence $\td[\bQ(f_{[G-e]}(p_t)):\bQ]=2n-3$, which gives a contradiction.

Now consider the subcase where $\iota$ is the reflection in the plane through $(0,1,0)$ and the $z$-axis. Then, for any realisation $(G,p)$, if $p(v_i)=(x_i,y_i,z_i)$ we have $p^{-1}(v_i)=(-x_i,y_i,z_i)$. Since $p_t(v_i)=p_t^{-1}(v_i)$ we have $p_t(v_i)=(0,y_i,z_i)$ for all $v_i \in V$. As before we have $\td[\bQ(p_t):\bQ]\leq n-1$. But $\td[\bQ(f_{[G-e]}(p_t)):\bQ]=2n-3$, which gives a contradiction.

Finally consider the subcase where $\iota$ is the composition of the reflection in the $z=0$ plane and in the plane through $(0,1,0)$ and the $z$-axis.
Recall that reflections generate a group, in this case $\bZ_2 \times \bZ_2$, so this is indeed the last case. Then, for any realisation $(G,p)$, if $p(v_i)=(x_i,y_i,z_i)$ we have $p^{-1}(v_i)=(-x_i,y_i,-z_i)$. Since $p_t(v_i)=p_t^{-1}(v_i)$ we have $p_t(v_i)=(0,\pm 1,0)$ for all $v_i \in V$. Here $\td[\bQ(p_t):\bQ]= 0$. But $\td[\bQ(f_{[G-e]}(p_t)):\bQ]=2n-3$, which gives a contradiction.

\smallskip
\noindent
{\bf Case 3}: $\M=\C$. Then $\iota$ is the reflection in the plane through $(0,y,y^2)$ and the $z$-axis. Then, for any realisation $(G,p)$, if $p(v_i)=(x_i,y_i,z_i)$ we have $p^{-1}(v_i)=(-x_i,y_i,z_i)$. Since $p_t(v_i)=p_t^{-1}(v_i)$ we have $p_t(v_i)=(0,y_i,z_i)$ for all $v_i \in V$. We have $\td[\bQ(p_t):\bQ]\leq n-1$. But $\td[\bQ(f_{[G-e]}(p_t)):\bQ]=2n-2$, which gives a contradiction.

\smallskip
\noindent
{\bf Case 4}: $\M=\E$. Then $\iota$ is a composition of reflections in the plane $x=0$, the plane $y=0$ and the plane $z=0$.
First consider the subcase when $\iota$ is the reflection in the plane $x=0$. Then, for any realisation $(G,p)$, if $p(v_i)=(x_i,y_i,z_i)$, we have $p^{-1}(v_i)=(-x_i,y_i,z_i)$. Since $p_t(v_i)=p_t^{-1}(v_i)$ we have $p_t(v_i)=(0,y_i,z_i)$ for all $v_i \in V$. We have $\td[\bQ(p_t):\bQ]\leq n$. But $\td[\bQ(f_{[G-e]}(p_t)):\bQ]=2n-1$, which gives a contradiction.
The other subcases follow similarly.
\end{proof}

The necessary conditions for global rigidity given in Proposition \ref{prop:kcon} and Theorem  \ref{thm:redundant} are independent since, for each $\M\in \{\S,\Y,\C,\E\}$, there are examples of  generic frameworks $(G,p)$ on $\M$ such that $G$ is $k$-connected for $k$ as in Proposition \ref{prop:kcon}, but $(G,p)$ is not redundantly rigid, and examples such that $(G,p)$ is redundantly rigid but not $k$-connected.


\section{Concluding Remarks}

\noindent1. Theorem \ref{thm:globalplane} and the result of Connelly and Whiteley \cite{C&W} that generic global rigidity in $\bR^2$ and the sphere are equivalent, shows
that the necessary conditions for generic global rigidity on the sphere given in Proposition \ref{prop:kcon} and Theorem  \ref{thm:redundant}
are sufficient. We conjecture that they are also sufficient for the cylinder and the cone.

\begin{con}\label{char}
Let $(G,p)$ be a generic framework on $\M$ where $\M\in \{\Y,\C\}$. Then $(G,p)$ is globally rigid on $\M$ if and only if either $G$ is a complete graph on at most four vertices or $G$ is $2$-connected and $(G,p)$ is redundantly rigid.
\end{con}

We know of no examples for which the necessary conditions  given in Proposition \ref{prop:kcon} and Theorem  \ref{thm:redundant}
fail to be sufficient to imply generic global rigidity on the ellipsoid. On the other hand
we do not even know how to  characterise generic rigidity  for the ellipsoid.

\smallskip
\noindent2.
If true, Conjecture \ref{char} would give a  polynomial algorithm to check generic global rigidity on the cylinder and cone:
redundant rigidity can be checked  in polynomial time, see for example \cite{L&S} or \cite{B&J2}, as can $2$-connectivity, see \cite{Tarjan}.

\smallskip
\noindent3. The methods in this paper can be adapted to prove analogues of Theorem \ref{thm:redundant} for a number of other surfaces including
tori, elliptical cylinders and elliptical cones. It is conceivable that the methods will apply to any irreducible $2$-dimensional algebraic variety embedded in $\bR^3$.

\smallskip
\noindent \textbf{Acknowledgements.} We would like to thank the School of Mathematics, University of Bristol for supporting the first author's visits during which most of this research took place.
We also thank Edward Crane for helpful conversations concerning generic points on algebraic varieties.

\vspace{-.5cm}

\bibliographystyle{abbrv}
\bibliography{JMNarxiv}

\end{document}